\documentclass[10pt]{amsart}

\usepackage{amssymb}

\usepackage{enumerate} % for using different counters
\usepackage{helvet} % for Sans Serif fonts
\usepackage{courier} % default ttdefault
\usepackage{eucal} % better calligraphic fonts
\usepackage{mathrsfs} % for fancy calligraphic fonts

\reversemarginpar %marginal notes appear on the left side margin
\normalmarginpar %switches back to right side margin

\let\oldmarginpar\marginpar %changes the font of margin text
\renewcommand\marginpar[1]{\-\oldmarginpar{\raggedright\small\sf #1}}

\title[Exceptional collections on  fake
  projective planes]{Exceptional collections on $2$-adically uniformised fake
  projective planes}

\author{Najmuddin Fakhruddin}
 
\address{School of Mathematics, Tata Institute of Fundamental Research, 
Homi Bhabha Road, Mumbai 400005, India}

\email{naf@math.tifr.res.in}

\newcommand{\nc}{\newcommand}

\nc{\rnc}{\renewcommand}

\nc{\bs}{\backslash}
\nc{\te}{\otimes}
\nc{\lf}{\lfloor} %for round down
\nc{\rf}{\rfloor}
\nc{\lc}{\lceil}  %for round up
\nc{\rc}{\rceil}
\nc{\lr}{\longrightarrow}
\nc{\sr}{\stackrel}
\nc{\dar}{\dashrightarrow}
\nc{\thra}{\twoheadrightarrow}

\nc{\mc}{\mathscr} 
\nc{\me}{\mathcal}
\nc{\mb}{\mathbb}
\nc{\mf}{\mathbf}
\nc{\mr}{\mathrm}
\nc{\mg}{\mathfrak}

\nc{\bP}{\mathbb{P}}
\rnc{\P}{\mathbb{P}}
\nc{\Q}{\mathbb{Q}}
\nc{\Z}{\mathbb{Z}}
\nc{\C}{\mathbb{C}}
\nc{\R}{\mathbb{R}}
\nc{\A}{\mathbb{A}}
\nc{\V}{\mathbb{V}}
\nc{\W}{\mathbb{W}}
\nc{\N}{\mathbb{N}}
\nc{\D}{\mathbb{D}}
\nc{\G}{\mathbb{G}}
\nc{\F}{\mathbb{F}}
\nc{\qb}{\overline{\mathbb{Q}}}

\nc{\del}{\partial}

\nc{\wt}{\widetilde}
\nc{\wh}{\widehat}
\nc{\ov}{\overline}

\nc{\aff}{{\A}^1}
\nc{\naive}{\!\sim_n}

\nc{\Spec}{\mr{Spec}}

\nc{\omx}{\omega_X}

\nc{\ep}{\epsilon}
\nc{\ve}{\varepsilon}
\nc{\vt}{\vartheta}

\nc{\ovl}{\ov{\lambda}}
\nc{\vl}{\mb{V}_{\ovl}}
\nc{\dl}{\mb{D}_{\ovl}}
\nc{\mnb}{\ov{\mr{M}}_{0,n}}
\nc{\mn}{\mr{M}_{0,n}}
\nc{\mel}{\ov{\mr{M}}_{1,1}}
\nc{\mfb}{\ov{\mr{M}}_{0,4}}
\nc{\mof}{\mr{M}_{0,4}}
\nc{\mgnb}{\ov{\mr{M}}_{g,n}}
\nc{\mgn}{\ov{\mr{M}}_{g,n}}
\nc{\omc}{\ov{\mr{M}}}

\rnc{\sl}{\shoveleft}

\nc{\res}{\operatorname{Res}}
\nc{\pic}{\operatorname{Pic}}
\nc{\spec}{\operatorname{Spec}}
\nc{\im}{\operatorname{Im}}
\nc{\gal}{\operatorname{Gal}}
\nc{\fr}{\operatorname{Fr}}
\nc{\ed}{\operatorname{ed}}
\nc{\rank}{\operatorname{rank}}
\nc{\h}{\operatorname{H}}
\nc{\ch}{\operatorname{char}}
\nc{\sw}{\operatorname{sw}}
\nc{\rsw}{\operatorname{rsw}}
\nc{\Mor}{\operatorname{Mor}}
\nc{\Per}{\operatorname{Per}}
\nc{\prep}{\operatorname{Prep}}
\nc{\End}{\operatorname{End}}
\nc{\Orb}{\operatorname{Orb}}

\newtheorem{thm}{Theorem}[section]

\newtheorem*{thmn}{Theorem}

\newtheorem{prop}[thm]{Proposition}
\newtheorem{conj}[thm]{Conjecture}

\newtheorem{lem}[thm]{Lemma}

\theoremstyle{definition}

\newtheorem{rem}[thm]{Remark}

\numberwithin{equation}{section}

\begin{document}

\begin{abstract}
  We show that there exist exceptional collections of length $3$
  consisting of line bundles on the three fake projective planes that
  have a $2$-adic uniformisation with torsion free covering group. We
  also compute the Hochschild cohomology of the right orthogonal of
  the subcategory of the bounded derived category of coherent sheaves
  generated by these exceptional collections.
\end{abstract}

\maketitle

\section{Introduction}

Mumford showed in \cite{mumford-fake} that discrete cocompact
torsion-free subgroups $\Gamma \subset \mr{PGL}_3(\Q_2)$ which act
transitively on the vertices of the Bruhat--Tits building of
$\mr{PGL}_{3,\Q_2}$ give rise to so called fake projective
planes---surfaces of general type with the same Betti numbers as
$\P^2$---and he gave one example of such a subgroup. All such
subgroups were classified by Cartwright, Mantero, Steger and Zappa
\cite{cmsz2}; there are two others and they give rise to two other
fake projective planes \cite{ishida-kato}.

More recently, all fake projective planes over $\C$ were classified by
Prasad and Yeung \cite{prasad-yeung} and Cartwright and Steger
\cite{cartwright-steger}. However, from the algebro-geometric point of
view these surfaces are still not well understood; for example, it is
still not known whether Bloch's conjecture on zero cycles on surfaces
with $p_g =0$ holds for any of these surfaces. Another question about
surfaces of general type with $p_g=0$ that has arisen very recently is
the existence of exceptional collections of maximal possible length in
their bounded derived categories of coherent sheaves. The first such
example was found by B\"ohning, von Bothmer and Sosna
\cite{bohning-bothmer-sosna} and subsequently several other examples
have been found. In the article \cite{GKMS} of Galkin, Katzarkov,
Mellit and Shinder, the authors conjecture that such exceptional
collections exist on all fake projective planes admitting a cube root
of the canonical bundle. The following is a consequence of the main result of this
paper, Theorem \ref{thm:exceptional}.

\begin{thmn}
  Let $M$ be a fake projective plane having a $2$-adic uniformisation
  with a torsion free covering group. There is an exceptional
  collection of length $3$ in $\me{D}^b(M)$ consisting of line
  bundles.
\end{thmn}

We note that $3$ is the smallest possible length of an exceptional
collection in $\me{D}^b(X)$, for $X$ any smooth projective variety, so
that the right orthogonal $\me{A}$ to the exceptional collection is a
quasiphantom subcategory of $\me{D}^b(X)$, i.e., such that the
Hochschild homology $HH_{\bullet}(\me{A}) = 0$. Using a spectral
sequence recently constructed by Kuznetsov \cite{kuznetsov-qph} which
has a very simple form in our setting, we are also able to compute the
Hochschild cohomology of $\me{A}$. While we do not prove Bloch's
conjecture for these surfaces, we formulate a general conjecture,
Conjecture \ref{conj:van}, on the $K_0$ of quasiphantom subcategories
which we hope might lead to a proof.

\smallskip

Our method of proof depends crucially on the fact that the fake
projective planes with $2$-adic uniformisations have natural regular
proper models over $\spec(\Z_2)$. The special fibre in all cases is an
explicit irreducible rational surface whose normalisation is the
blowup of $\P^2_{\F_2}$ along its rational points. The main technical
result is the computation of all the cohomology groups of a natural
class of line bundles on these fake projective planes, Proposition
\ref{prop:cohom}. The particular case of this relevant to the
construction of exceptional collections is proved by using a Galois
theoretic argument and specialisation, eventually reducing this to an
explicit computation on $\P^2_{\F_2}$.

\smallskip 

After the first version of this paper was put on arXiv, we learned
from L.~Katzarkov that the authors of \cite{GKMS} had proved their
conjecture for $6$ fake projective planes over $\C$, distinct from the
ones we have considered, and by different methods. This is included in
v2 of \cite{GKMS}.

\section{Line bundles on fake projective planes}

\subsection{}
We denote by $\mc{X}$ the formal scheme over $\spec(\Z_2)$
corresponding to $\mr{PGL}_3(\Q_2)$ constructed by Mustafin
\cite{mustafin} and Kurihara \cite{kurihara}; the reader may also
consult \cite{mumford-fake} for an exposition in the case we use. The
irreducible components of the special fibre of $\mc{X}$ are in
bijection with the vertices of the Bruhat--Tits builiding of
$\mr{PGL}_{3,\Q_2}$ and each of these components is isomorphic to the
surface $B$ obtained by blowing up $\P^2_{\F_2}$ at all its
$\F_2$-rational points. There is a faithful action of
$\mr{PGL}_3(\Q_2)$ on $\mc{X}$ which is transitive on the irreducible
components of the special fibre and the stabilizer of each component
is isomorphic to $\mr{PGL}_3(\Z_2)$. This action restricts to a
faithful action on $\widehat{\Omega}^2$, the two dimensional Drinfeld
upper half space over $\Q_2$, which is the generic fibre (as a rigid
analytic space over $\Q_2$) of $\mc{X}$. Moreover,
$\widehat{\Omega}^2$ is an admissible open subset (in the sense of
rigid analytic geometry) of $\P^2_{\Q_2}$ and the $\mr{PGL}_3(\Q_2)$
action on it is compatible with this inclusion and the natural action
of $\mr{PGL}_3(\Q_2)$ on $\P^2_{\Q_2}$.

If $\Gamma$ is a discrete torsion-free cocompact subgroup of
$\mr{PGL}_3(\Q_2)$, then one may form the quotient formal scheme
$\mc{X}/\Gamma$. The dualising sheaf $\omega_{\mc{X}}$ descends to a
line bundle on $\mc{X}/\Gamma$ which is ample on the special fibre,
hence by Grothendieck's existence theorem, $\mc{X}/\Gamma$ is the
formal completion of a unique regular projective scheme over
$\spec(\Z_2)$. If $\Gamma$ acts transitively on the irreducible
components of $\mc{X}$, then Mumford shows that the generic fibre $M$
of $\mc{M}$, the projective scheme over $\spec(\Z_2)$ corresponding to
$\mc{X}/\Gamma$, is a fake projective plane. The special fibre $M_0$
of $\mc{M}$ is an irreducible surface over $\F_2$ whose normalisation
is isomorphic to $B$.

\begin{lem} \label{lem:nonexist} Let $F$ be any finite unramified
  extension of $\Q_2$.  For all line bundles $L$ on $M_F$ we have $9
  \mid c_1(L)^2$. In particular, $\omega_M$ does not have a cube root
  defined over $F$.
\end{lem}

\begin{proof}
  Let $A_F$ be the ring of integers of $F$. Since $\mc{M}$ is regular
  and $F$ is unramified, it follows that $\mc{M}_F := \mc{M}
  \otimes_{\Z_2} A_F$ is also regular. Since $M_0$ is geometrically
  irreducible, it follows that the restriction map $\pic(\mc{M}_F) \to
  \pic(M_F)$ is an isomorphism. In particular, any line bundle $L$ on
  $M_F$ extends uniquely to a line bundle $\mc{L}$ on
  $\mc{M}_F$. Since $M$ is a fake projective plane, $\pic(M_F)$ modulo
  its (finite) torsion subgroup is isomorphic to $\Z$. Since
  $c_1(\omega_M)^2 = 9 \neq 0$, it follows that there exists a
  positive integers $m, n$ so that $\mc{L}^{\otimes m}$ is isomorphic
  to $\omega_{\mc{M}}^{\otimes n}$.

  From the computations of \cite[p.~238]{mumford-fake}, it follows
  that the degree of $\omega_{M_0}$, which is the restriction of
  $\omega_{\mc{M}}$ to $M_0$, on the image of any exceptional divisor
  in $B$ is $1$. Since the degree of $\mc{L}$ on the curve must also
  be an integer, it follows from $\mc{L}^{\otimes m} \cong
  \omega_{\mc{M}}^{\otimes n}$ that $m \mid n$, so $9 =
  c_1(\omega_M)^2 \mid c_1(L)^2$.
\end{proof}

\subsection{}

We now show that $\omega_M$ does have cube roots defined over certain
cubic extensions of $\Q_2$. The existence of cube roots over some
extension also follows from the classification results of Prasad and
Yeung \cite{prasad-yeung}, and the basic principle of our proof is the
same. However, the argument below is elementary and is essentially
immediate from the construction of the groups $\Gamma$. More
importantly, it also gives precise information about the field of
definition of the cube roots.

Let $\ov{\Q}_2$ be an algebraic closure of $\Q_2$. There is a natural
surjection $q:\mr{SL}_3(\ov{\Q}_2) \to \mr{PGL}_3(\ov{\Q}_2)$; we let
$G$ denote the group $q^{-1}( \mr{PGL}_3({\Q}_2))$ so that there is
a short exact sequence
\[
1 \to \mu_3(\ov{\Q}_2) \to G \to \mr{PGL}_3({\Q}_2) \to 1 \ .
\]
Suppose $\Gamma'$ is a subgroup of $G$ mapping isomorphically onto
$\Gamma$ by $q$ and so that all elements of $\Gamma'$ are defined over
a finite extension $k$ of $\Q_2$. Then $\Gamma'$ acts on
$\widehat{\Omega}^2 \otimes_{\Q_2}k$ via $q$ and the base change of
the action of $\Gamma$ on $\widehat{\Omega}^2$. We denote by $O(-1)$
the inverse of the standard generator of $\pic(\P^2_{\Q_2})$ as well
as its restriction to $\widehat{\Omega}^2$ and $\widehat{\Omega}^2
\otimes_{\Q_2}k$. Since the line bundle $O(-1)$ on $\P^2_k$ has a
natural $\mr{SL}_{3,k}$ linearisation, the inclusion of $\Gamma'$ in
$\mr{SL}_3(k)$ gives rise to a linearisation of $O(-1)$ on
$\widehat{\Omega}^2 \otimes_{\Q_2}k$, so it descends to a line bundle
$L$ on $M_k =(\widehat{\Omega}^2 \otimes_{\Q_2}k)/ \Gamma $. Since
$\omega_M$ is the line bundle on $M$ corresponding to $O(-3)$ with its
induced linearisation, it follows that $L^{\otimes 3} \cong
\omega_{M_k}$.

\subsection{}
We now check that subgroups $\Gamma'$ as above exist for all the three
fake projective planes and also determine the extensions $k$
corresponding to these subgroups. This requires explicit knowledge of
the groups $\Gamma$, so we have to consider Mumford's example and the
CMSZ examples separately. However, in both cases $\Gamma$ is contained
in a larger lattice $\Gamma_1$ which can be lifted to a lattice
$\Gamma_1'$ in $G$ in a very simple way.

\subsubsection{The Mumford lattice}
Mumford's lattice $\Gamma$ is a sublattice of index $21$ of the
subgroup $\Gamma_1$ of $\mr{PGL}_3(\Q_2)$ generated by the images of
the matrices
\[
\sigma = 
\begin{bmatrix}
1 & 0 & \lambda \\
0 & 0 & -1 \\
0 & 1 & -1
\end{bmatrix}
\ \ \ \ , \ \ \ 
\tau = 
\begin{bmatrix}
0 & 0 & 1 \\
1 & 0 & 1 + \lambda \\
0 & 1 & \lambda
\end{bmatrix}
\ \ \ \ \ \mathrm{and} \ \ \ 
\rho = 
\begin{bmatrix}
1 & 0 & \lambda \\
0 & 1 & -\lambda^3/2 \\
0 & 0 & \lambda^2/2
\end{bmatrix}
\]
where $\lambda$ is a certain element of $\Q_2$ of the form $2u$ with
$u$ a unit \cite[\S 2]{mumford-fake}. One sees immediately that
$\sigma, \tau \in \mr{SL}_3(\Q_2)$ while $\det(\rho) = \lambda^2/2$
has valuation $1$. Let $\mu$ be a cube root of $\det(\rho)$, $k =
\Q_2(\mu)$ and $\rho' = \mu^{-1} \rho$. Clearly $\rho' \in \mr{SL}_3(k)$
and the image in $\mr{PGL}_3(k)$ of the subgroup $\Gamma_1'$ of
$\mr{SL}_3(k)$ generated by $\sigma$, $\tau$ and $\rho'$ is equal to
$\Gamma_1$. Moreover, since $k$ does not contain a primitive cube root
of $1$, the only scalar matrix in $\Gamma_1$ is the identity. It
follows that $\Gamma_1'$ maps isomorphically onto $\Gamma_1$. We then
let $\Gamma'$ be the inverse image of $\Gamma$ in $\Gamma_1'$. The
three choices for $\mu$ give rise to $3$ such subgroups
$\Gamma_1'$. The three subgroups $\Gamma'$ which they give rise to are
also distinct since, by Lemma \ref{lem:nonexist}, none of the
$\Gamma'$ can be subgroups of $\mr{SL}_3(\Q_2)$.

\subsubsection{The CMSZ lattices}

The lattices constructed by Cartwright, Mantero, Steger and Zappa
\cite[p.~181]{cmsz2} are both sublattices of $\Gamma_1$, the image of the
subgroup of $\mr{GL}_3(\Q_2)$ generated by the elements
\[
a_3 =
\begin{bmatrix}
0 & 0 & -(S-1)/4 \\
1 & 0 & 1 \\
0 & 1 & (S-1)/4
\end{bmatrix}
\ \ \ \mr{and} \ \ \ 
s = 
\begin{bmatrix}
0 & -1 & -(S-1)/4 \\
1 & -1 & -(S-5)/4 \\
0 & 0 & 1
\end{bmatrix}
\]
where $S\in \Z_2$ is the square root of $-15$ which is congruent to
$1$ modulo $4$ \cite[p.~182]{cmsz2}. Clearly, $s \in \mr{SL}_3(\Q_2)$
while $\det(a_3) = (S-1)/4$.  Since $(S-1)(S+1) = -16$, it follows
that in fact the valuation of $S-1$ is $3$ and so $(S-1)/4$ is a
uniformizer. By letting $k$ be the extension of $\Q_2$ obtained by
adjoing a cube root as above and then modifying $a_3$, it follows as
in the previous case that we get three distinct lifts of $\Gamma$ (for
both choices of $\Gamma$).

\subsubsection{}
In each of the cases discussed above, it can be seen that $\mr{Hom}
(\Gamma,\mu_3)$ has order three, so the three lifts that we have
constructed are in fact all.

\section{Cohomology of line bundles}

Henceforth, $M$ denotes any one of the fake projective planes
considered earlier.  We let $K$ be the Galois closure of any of the
cubic extensions $k$ of $\Q_2$ of the previous section, so it is a
Galois extension of $\Q_2$ with Galois group $S_3$, containing the
unramified quadratic extension $F = \Q_2(\zeta)$, with $\zeta^2 +
\zeta + 1 =0 $; the extension $K/F$ is totally ramified. We will
compute the dimensions of the cohomology groups of all line bundles on
$M_K$ contained in the subgroup $P$ of $\pic(M_K)$ generated by the
cube roots of $\omega_{M_K}$ constructed above and the line bundles of
order two coming from characters of $\Gamma$ in $\Q_2^{\times}$. Using
the explicit description of the groups $\Gamma$ given in
\cite{mumford-fake} and \cite{cmsz2} or the figures at the end of
\cite{ishida-kato}, one can see that this group is isomorphic to $\Z
\times \Z/3 \times (\Z/2)^2$ for the Mumford surface and one of the
CMSZ surfaces and to $\Z \times \Z/3$ for the other CMSZ surface. We
do not use these computations in the sequel so we do not give the
details and the arguments that follow do not depend on a case by case
analysis of the surfaces.

We define the degree of $L$, $\deg(L)$, to be the positive square
root of $c_1(L)^2$ if $L$ is ample and its negative otherwise. Since
$\omega_M$ is ample, in order to compute the cohomology of all line
bundles it suffices, by Serre duality, to consider only line bundles
$L$ which are ample.

\begin{prop} \label{prop:cohom} For any $L \in P$ let $h^i(L)$ denote
  the dimension of $H^i(M_K,L)$.  Let $L \in P$ be ample and let $d =
  \deg(L)$.
\begin{enumerate}
\item If $d \in \{1,2\}$, then $h^i(L) = 0$ for all $i$. 
\item If $d = 3$, then $h^i(L) = 0$ for $i=0,1$ and $h^2(L) =
  1$ if $L \cong \omega_M$ else $h^0(L) = 1$ and $h^i(L) = 0$
  for $i=1,2$. 
\item If $d>3$, then $h^0(L) = (d-1)(d-2)/2$  and
  $h^i(L) = 0$ for $i=1,2$.
\end{enumerate}
Furthermore, $h^1(L) = 0$ for any $L \in P$.
\end{prop}

It seems reasonable to expect that a similar result holds for all line
bundles on all fake projective planes.

\begin{proof}
  Since $c_1(\omega_M)^2 = 9$, we have $c_1(L)\cdot c_1(\omega_M) =
  3d$. Moreover, $\chi(O_M) = 1$, so by the Riemann--Roch theorem for
  surfaces we have $\chi(L) = (d-1)(d-2)/2$.

  If $d>3$ then $L \otimes \omega^{-1}$ is ample so the result in this
  case follows from the Kodaira vanishing theorem.

  % Suppose $d=2$ and $h^0(L) > 1$. Since $d=2$, $L$ is not defined over
  % the maximal unramified extension of $\Q_2$ (which has trivial Brauer
  % group $M$) by Lemma \ref{lem:nonexist}, so it has a Galois conjugate
  % $L'$ which is not isomorphic to $L$. Since $L'$ is a Galois
  % conjugate, it is also ample, has $d=2$ and also $h^0(L') = h^0(L) >
  % 1$. The image of the multiplication map $H^0(M_K,L) \otimes
  % H^0(M_K,L') \to H^0(M_K,L\otimes L')$ has dimension at least $4$
  % since $L$ and $L'$ are not isomorphic. But $\deg(L \otimes L') =4$,
  % so by the $d>3$ case already considered above we have $h^0(L \otimes
  % L') = 3$ which is a contradiction.

  We now consider the case $d = 2$ and assume that $h^0(L) > 0$. By
  the definition of $P$, $L$ is isomorphic to $L_1^{\otimes 2} \otimes
  T$, where $L_1$ is one of the cube roots of $\omega_{M_K}$
  constructed in \S 2 and $T$ is either a trivial line bundle or has
  order $2$. The action of $\gal(K/\Q_2)$ on $P$ preserves all the
  bundles of order $2$ and permutes the three cube roots of
  $\omega_{M_K}$, so it follows that $L$ has two other Galois
  conjugates, say $L'$ and $L''$, and $L \otimes L' \otimes L'' \cong
  \omega_{M_K}^{\otimes 2} \otimes T$.

  Let $k/\Q_2$ be the cubic extension over which $L$ is defined and
  let $D$ be any divisor (defined over $k$) in the linear system
  corresponding to $L$. It follows that $D$ has two Galois conjugates,
  $D'$ and $D''$, such that $O(D') \cong L'$ and $O(D'') \cong
  L''$. Then $C_K := D + D' + D''$ is a Galois invariant divisor in
  the linear system corresponding to $\omega_{M_K}^{\otimes 2} \otimes
  T$, so it is the base change of a divisor $C$ on $M$.  Let $C_0$ be
  the specialisation of $C$ in $M_0$. Since $T$ corresponds to a
  character of $\Gamma$ of order at most $2$, the specialisation of
  $T$ is trivial, hence $C_0$ is a Cartier divisor in the linear
  system corresponding to $\omega_{M_0}^{\otimes 2}$.

  We claim that the Weil divisor associated to $C_0$ is divisible by
  $3$ in the group of Weil divisors on $M_0$. It suffices to prove
  this over $M_{0,\F_4}$, and since $F/\Q_2$ is unramified (so
  specialisation commutes with base change) it is enough to consider
  the specialisation of $C_F$ as a divisor on $M_F \subset \mc{M}_R$,
  where $R$ is the ring of integers in $F$.

  Observe that no prime divisor in the support of $D$ can be preserved
  by $\gal(K/F) \cong \Z/3$, since any such divisor would descend to a
  divisor of degree $1$ or $2$ defined over $F$ which is not possible
  by Lemma \ref{lem:nonexist} (since $F/\Q_2$ is unramified). It
  follows that each prime divisor $Z$ in the support of $C_F$ splits
  into a sum of three prime divisors over $K$. We show that the
  specialisation of any such $Z$ has multiplicity $3$ along each
  component of its support.

  Let $\mc{Z}$ be the Zariski closure of $Z$ in $\mc{M}_R$ and
  $\wt{\mc{Z}}$ its normalisation. Since $Z$ splits into $3$
  components over $K$, the function field of $Z$ must contain $K$.
  Thus, since $\wt{\mc{Z}}$ is normal, the morphism $\wt{\mc{Z}} \to
  \spec(R)$ factors though $\spec(S)$, where $S$ is the ring of
  integers of $K$. The specialisation of $\wt{\mc{Z}}$ is given by the
  valuations of a uniformizer of $R$ with respect to the discrete
  valuation corresponding to the generic point of each irreducible
  component of the closed fibre. Since $K/F$ is ramified and of degree
  $3$, so a uniformizer in $R$ is (up to a unit) the cube of a
  uniformizer in $S$, it follows that each irreducible component of
  $\wt{\mc{Z}} \times_R \F_4$ has multiplicity divisible by $3$. Since
  specialisation commutes with proper pushforward \cite[Proposition
  20.3]{fulton-it}, the same holds for the irreducible components of
  $\mc{Z} \times_R \F_4$, thereby proving the claim.

  From the claim we see that all the irreducible components of $C_0$
  have multiplicity divisible by $3$ in the corresponding Weil
  divisor. By Lemma \ref{lem:three} below no such divisor exists.
  Thus, if $d=2$, we must have $h^0(L) = 0$.

  If $d=1$ and $h^0(L) \neq 0$, then $h^0( L^{\otimes 2}) \neq 0$.  It
  follows from the $d=2$ case already considered that this is not
  possible.

  If $d \in \{1,2\}$, then $L^{\otimes -1} \otimes \omega_{M_K}$ has
  degree $3-d \in \{1,2\}$, so it follows from the above and Serre
  duality that $h^2(L) = 0$. Since $\chi(L) = 0$, it follows that we
  also have $h^1(L) = 0$.

  If $d=3$ and $L \cong \omega_M$, then the statements follow since
  $p_g(M) = q(M) = 0$. Otherwise, $L_0 = L \otimes \omega_M^{-1}$ is a
  non-trivial torsion line bundle, so by Serre duality $h^2(L) =
  0$. Since $\chi(L) = 1$, it follows that $h^0(L) > 0$. Since $L \in
  P$, the line bundle $L_0$ corresponds to a non-trivial homomorphism
  $f$ from $\Gamma$ into $K^{\times}$. Let $\pi:M' \to M_K$ be the
  cover of $M_K$ corresponding to $\mr{Ker}(f)$. Then $\pi_*(O_{M'})$
  is isomorphic to $\oplus_{i=0}^{e-1} (L_0)^{\otimes i}$, where $e$
  is the order of $L_0$ in $\pic(M_K)$. From the projection formula,
  $\pi_*(\omega_{M'})$ is isomorphic to $\oplus_{i=0}^{e-1}\omega_M
  \otimes (L_0)^{\otimes i}$. We have $\chi(\omega_{M'}) = e
  \chi(\omega_M)$ and $q(M') = 0$ by \cite[p.~238]{mumford-fake}, so
  $p_g(M') = e-1$. It follows that $\sum_{i=0}^{e-1} h^0(\omega_M
  \otimes (L_0)^{\otimes i}) = e-1$. Since all summands except for
  $i=0$ must be at least $1$, it follows that they are all equal to
  $1$. In particular, $h^0( \omega_M \otimes L_0) = h^0(L) = 1$. Since
  $\chi(L)$ is also $1$, it follows that $h^1(L) = 0$.

  The last statement follows from the previous claims if $L$ is
  ample. If $L$ is not ample, then $\omega_{M_K}\otimes L^{-1}$ is
  ample so the claim follows by Serre duality.
\end{proof}

\begin{lem} \label{lem:three} There is no (Cartier) divisor $C_0$ in the
  linear system on $M_0$ associated to $\omega_{M_0}^{\otimes 2}$ such
  that the associated Weil divisor has multiplicity divisible by $3$
  along each geometric component of its support.
\end{lem}

\begin{proof}
  Suppose such a divisor $C_0$ exists.  Let $\nu: B \to M_0$ be the
  normalisation morphism and $\pi: B \to \P^2_{\F_2}$ the morphism
  blowing up all $\F_2$-rational points of $\P^2_{\F_2}$. The morphism
  $\nu$ identifies each exceptional curve $E_i$ with the strict
  transform of a $\F_2$-rational line $F_i$ in $\P^2_{\F_2}$ in a way
  that depends on the particular fake projective plane under
  consideration. According to \cite[p.~238]{mumford-fake},
  $\nu^*(\omega_{M_0})$ is equal to $O_B(\pi^*4H -\sum_{i=1}^7E_i)$,
  where $H$ is the class of a line in $\P^2_{\F_2}$. It follows that
  $\wt{C}:=\nu^*(C_0) = \pi^*(C_1) - \sum_{i=1}^72E_i$ (as Cartier
  divisors) where $C_1$ is a curve of degree $8$ in $\P^2_{\F_2}$
  passing through all the $\F_2$-rational points and having
  multiplicity at least $2$ at each such point.

  If $\nu(E_i) \subset M_0$ is contained in the support of $C_0$ for
  some $i$, then both $E_i$ and $F_i$ must be contained in the support
  of $\wt{C}$. Moreover, since the multiplicity of this component is
  divisible by $3$, the sum of the multiplicities of $E_i$ and $F_i$
  in $\wt{C}$ must be divisible by $3$.

  First suppose that the support of $C_0$ is not contained in the
  double point locus of $M_0$. Then $C_1$ has a component which is not
  a rational line and the multiplicity of each such component must be
  be divisible by $3$; let $C_1'$ be the union of all such components
  (with multiplicity). Since $\deg(C_1) = 8$, it follows that
  $\deg(C_1')$ is $3$ or $6$. If $\deg(C_1') = 3$, then $C_1'$ must be
  a triple (rational) line which is not possible by assumption. If
  $\deg(C_1') = 6$ then the corresponding reduced curve must be a
  conic and $C_1'' := C_1 - C_1'$ is either a double line or a union
  of two distinct lines.

  Suppose $C_1''$ is a double line. Since $C_1$ must contain all
  rational points it follows that $C_1'$ must contain at least $4$
  rational points. Since a smooth conic over $\F_2$ has only $3$
  rational points and a singular (irreducible) conic has only $1$
  rational point, it follows that $C_1'$ must be a union of two
  rational lines which is a contradiction.

  Suppose $C_1''$ is a union of two distinct rational lines, say $F_i$
  and $F_j$. It then follows that both $E_i$ and $E_j$ must have
  multiplicity at least $2$ in $\wt{C}$, so the corresponding points
  $p_i$ and $p_j$ in $\P^2_{\F_2}$ have multiplicity at least $4$ in
  $C_1$. Since the union of $F_i$ and $F_j$ contains $5$ rational
  points, $C_1'$ must be a smooth conic. Since a smooth conic contains
  $3$ rational points, it follows that $ (F_i \cup F_j) \cap C_1'$
  contains at most $1$ rational point. Since $F_i$ and $F_j$ have
  multiplicity $1$ in $C_1$, it follows that there is at most one
  point of multiplicity at least $4$ on $C_1$, a contradiction.

  It remains to consider the case that $C_0$ is contained in the
  singular locus of $M_0$, so $C_1$ is a union of rational lines (with
  multiplicity). Since every rational point must lie on $C_1$, $C_1$
  must have at least three irreducible components.

  If there are exactly three components, then there must be a (unique)
  point $p$ contained in all of them since this is the only way that
  the union of three lines in $\P^2_{\F_2}$ can contain all rational
  points. Since $\deg(C_1) =8$, $C_1$ has multiplicity $8$ at $p$, so
  the exceptional divisor $E_p$ has multiplicity $6 > 0$, in
  $C$. It follows that the strict transform of the corresponding line $F_p$
  must also be in the support of $\wt{C}$, so $F_p$ must be one
  of the lines in the support of $C_1$ and its multiplicity is
  divisible by $3$. If the multiplicity is $3$, then since $\deg(C_1)
  = 8$ one of the other lines must also have multiplicity $\geq
  3$. But then the $5$ points on the union of these lines will have
  multiplicity $>2$ in $C_1$, so the corresponding exceptional
  divisors must have multiplicity $>0$ in $C$. But this
  implies that $C_1$ must contain the $5$ rational lines corresponding
  to these exceptional divisors which is a contradiction. If the
  multiplicity of $F_p$ in $C_1$ is $6$ then both the other lines must
  have multiplicity one, but then there would exist rational points on
  $C_1$ of multiplicity $1$ which is not possible. Thus $C_1$ cannot
  have three components.

  Suppose $C_1$ has exactly $4$ components. Since all rational points
  must lie on $C_1$, one sees that there is only one such
  configuration (up to automorphisms of $\P^2_{\F_2}$) consisting of
  the union of all three lines passing through a distinguished point
  together with one other line $F$. It follows that exactly one
  rational point lies on $3$ lines, $3$ rational points lie on $2$
  lines each and the remaining $3$ rational points on a single line
  each. These three lines must have multiplicty at least $2$, since
  the multiplicity of any rational point on $C_1$ must be at least
  $2$. Furthermore, these three lines intersect $F$ in distinct
  points, so the rational points $p_1$, $p_2$ and $p_3$ on this line
  have multiplicity $>2$ on $C_1$. It follows that the exceptional
  divisors $E_{p_1}$, $E_{p_2}$ and $E_{p_3}$ corresponding to these
  points must be contained in the support of $\wt{C}$, so the
  corresponding lines $F_{p_1}$, $F_{p_2}$ and $F_{p_3}$ must be
  contained in the support of $C_1$. If $F$ has multiplicity $2$, then
  the multiplicity of $C_1$ at all $p_i$, $i=1,2,3$ is exactly
  $4$. But then the multiplicity of each $E_{p_i}$ in $C$ is $2$ so
  the multiplicity of each $F_{p_i}$ in $C_1$ must be congruent to $1$
  modulo $3$. Since $2 \not\equiv 1 \mod 3$ this is a contradiction.

  It follows that $F$ must have multiplicity $1$ and so one of the
  other lines must have multiplicity $3$. But then there are at least
  $5$ points on $C_1$ with multiplicity at least $3$ which implies
  that $C_1$ has at least $5$ lines in its support, also a
  contradiction.

  Suppose that there are $5$ lines in the support of $C_1$ so there
  are exactly $5$ rational points on $C_1$ with multiplicity $>2$.
  The union of any three lines contains at least $6$ rational points
  so at most $2$ of the lines are multiple and the multiplicities must
  be one line of multiplicity $2$ and another of multiplicity $3$ or a
  single line of multiplicity $4$.

  There is a unique configuration of $5$ lines up to automorphism. In
  such a configuration, there is one point lying on a single line, $4$
  on two lines each and the remaining $2$ lie on $3$ lines. The line
  $F$ containing the point $p$ which lies on only one line must be
  multiple, since the multiplicity of each point must be at least
  $2$. If the multiplicity is $4$, then there is no other multiple
  component, so there are $4$ points of multiplicity $2$ which is not
  possible.

  If the multiplicity of $F$ is $2$, then there is another component
  of multiplicity $3$. Then there are still three rational points of
  multiplicity $2$ which is not possible.

  If the multiplicity of $F$ is $3$, then there is another component
  of multiplicty $2$. Then the multiplicities of the points are
  $2,2,3,3,3,5,6$ so the multiplicities of the exceptional divisors
  are $0,0,1,1,1,3,4$. Since the sum of the multiplicity of any
  exceptional divisor and the line corresponding to it is divisible by
  $3$, it follows that there must be $4$ multiple lines which is a
  contradiction.

  Suppose there are $6$ lines in the support of $C_1$. Then there are
  $3$ collinear points which lie on $2$ lines each and the remaining
  $4$ points lie on $3$ lines each. Also, there must be exactly $6$
  rational points with multiplicity at least $3$ on $C_1$. Thus, there
  are exactly two points in the $3$ element set of points lying on
  only $2$ lines and a line containing each point of multiplicity
  $2$. The intersection point of these two lines will then have
  multiplicity $5$ and the two other points on them have
  multiplicity $4$. Thus, the multiplicities of the rational points on
  $C_1$ must be $2,3,3,3,4,4,5$ which, using the congruence modulo
  $3$ as above, implies that there must be at least $3$ lines of
  multiplicity $>1$ which is a contradiction.

  Finally, we consider the case that the support of $C_1$ is the union
  of all $7$ rational lines. Since each rational point lies on $3$
  lines and exactly one of the lines, call it $F$, must be double, the
  $4$ rational points not on $F$ have multiplicity $3$ on $C_1$. By
  the congruence argument as before, this implies that there must be
  $4$ multiple lines, a contradiction.

\end{proof}

\section{Exceptional collections}

Let $X$ be a smooth projective variety over a field $K$. A sequence of
objects $E_1,E_2,\dots,E_n$ of $\me{D}^b(X)$, the bounded derived
category of coherent sheaves on $X$, is called an exceptional
collection if $\mr{Hom}(E_j,E_i[k])$ is non-zero for $j \geq i$ and $k
\in \Z$ iff $i=j$ and $k=0$, in which case it is one
dimensional. Galkin, Katzarkov, Mellit and Shinder have conjectured
\cite[Conjecture 3.1]{GKMS} that if $X$ is an $n$-dimensional fake
projective space over $\C$ such that the canonical bundle $\omega_X$
has an $(n+1)$-th root $O_X(-1)$, then the line bundles $O_X, O_X(-1),
\dots, O_X(-n)$ form an exceptional collection. If $X$ is a surface,
they observe that to prove this it suffices to show that
$H^0(X,O_X(2)) = 0$. This conjecture appears to be difficult to prove
in general, but the computations of the previous section lead to the
following:

\begin{thm} \label{thm:exceptional}
\begin{enumerate}[(a)]
\item  Let $M$ be a $2$-adically
  uniformised fake projective plane over $\Q_2$ and let $L_1, L_2 \in
  P$ be line bundles of degree $-1$ and $-2$. Then the sequence of
  line bundles $O_{M_K},L_1,L_2$ is an exceptional
  collection.
\item Let $\me{B} = \langle O_{M_K},L_1,L_2\rangle$, the subcategory
  of $\me{D}^b(X)$ generated by $O_{M_K}$, $L_1$ and $L_2$ and let $\me{A}
  = \me{B}^{\perp}$. Then $HH_{\bullet}(\me{A}) = 0$, i.e., $\me{A}$
  is a quasiphantom category. Moreover, the dimensions of the vector
  spaces $HH^{t}(\me{A})$, $t \geq 0$, are given by the sequence
  $1,0,0,28,54,27,0,0,\dots$. In particular, the product of any two
  elements of $HH^{\bullet}(\me{A})$ of positive degree is $0$.
\end{enumerate}
\end{thm}

\begin{proof}
  For ease of notation, we denote $O_{M_K}$ by $L_0$. Then
  $\mr{Hom}(L_j,L_i[k]) = H^k(M_K, L_i \otimes L_j^{\otimes -1})$. If
  $j>i$, then the degree of $ L_i \otimes L_j^{\otimes -1}$ is $1$ or
  $2$ so the cohomology groups vanish for all $k$ by Proposition
  \ref{prop:cohom}. If $i=j$, then $ L_i \otimes L_j^{\otimes -1} \cong
  O_{M_K}$, and since $p_g(M) = q(M) = 0$, (a) follows.

  The first part of (b) follows from the fact that the Betti numbers
  $b_i(M)$ are equal to $1$ for $i=1,2,4$ and $0$ for all other $i$
  together with Kuznetsov's additivity theorem for Hochschild
  homology \cite[Theorem 7.3]{kuznetsov-hh}.

  To compute the Hochschild cohomology of $\me{A}$ we first compute that
  of $\me{D}^b(X)$. Recall, for example from \cite{kuznetsov-hh}, that
  this is given by
  \[
  HH^t(\me{D}^b(X)) = \bigoplus_{p=0}^n H^{t-p}(M, \wedge^p T_M) \ .
  \]
  The only nonzero cohomology group of $O_M$ is in degree $0$ where
  the dimension is $1$. We have $H^0(M,T_M)=0$ since $M$ is of general
  type and $H^1(M,T_M) = 0$ since $M$ is (infinitesimally) rigid. By
  the Hirzebruch-Riemann-Roch theorem we then have
\begin{multline*}
  h^2(M, T_M) = \chi(T_M) = \int_M ch(T_M)\cdot td(T_M) \\
  = \int_M (2+c_1(T_M) + (c_1(T_M)^2 -2c_2(T_M))/2)\cdot (1 +
  c_1(T_M)/2 + (c_1(T_M)^2 +
  c_2(T_M))/12) \\
  = \int_M (c_1(T_M)^2 + c_2(T_M))/6 + c_1(T_M)^2/2 + (c_1(T_M)^2
  -2c_2(T_M))/2 = 8
\end{multline*}
where we use that $c_1(T_M)^2 = 9$ and $c_2(T_M) = 3$.

The only non-zero cohomology group of $\wedge^2T_M = \omega_M^{-1}$ is
in degree $2$ and $h^2(M,\omega_M^{-1}) = h^0(M, \omega_M^2) = 10$ by
Proposition \ref{prop:cohom}. Thus, the dimensions of the vector
spaces $HH^t(\me{D}^b(X))$, for $t \geq 0$, are given by the sequence
$1,0,0,8,10,0,0,\cdots$.

In the article \cite{kuznetsov-qph}, Kuznetsov defines the normal
Hochschild cohomology of $\me{B}$ in $\me{D}^b(X)$, denoted by
$NHH^{\bullet}(\me{B}, \me{D}^b(X))$, for $\me{B}$ any admissible
subcategory of $\me{D}^b(X)$ with $X$ a smooth projective variety, which
sits in a distinguished triangle
\begin{equation} \label{eq:seq}
NHH^{\bullet}(\me{B}, \me{D}^b(X)) \to HH^{\bullet}(\me{D}^b(X)) \to
HH^{\bullet}(\me{A}) \sr{+1}{\lr} \ ,
\end{equation}
where $\me{A} = \me{B}^{\perp}$.  Moreover, if $\me{B}$ is generated
by an exceptional collection $E_1,E_2,\dots,E_n$ then he constructs a
spectral sequence converging to $NHH^{\bullet}(\me{B}, \me{D}^b(X))$
whose $\mf{E}_1^{-p,q}$ term is given by
\[
\bigoplus_{\substack{1 \leq a_0 <a_1 \cdots <a_p \leq n, \\ k_0 + \cdots + k_p =
  q}} Ext^{k_0}(E_{a_0},E_{a_1})\otimes \cdots \otimes
Ext^{k_{p-1}}(E_{a_{p-1}},E_{a_p}) \otimes
Ext^{k_p}(E_{a_p},S^{-1}E_{a_0})
\]
where $S^{-1}$ denotes the inverse of the Serre functor on
$\me{D}^b(X)$. The differentials in this spectral sequence are given by
(higher) multiplication maps, but in our situation it degenerates for
trivial reasons.

We apply the above with $X = M$, so $n=3$ and $(E_1,E_2,E_3) =
(L_0,L_1,L_2)$. The only possibilities for $p$ are $0,1,2$ and we see
from Proposition \ref{prop:cohom} that to get a non-zero summand in
the spectral sequence we must have $k_i = 2$ for $i<p$ and $k_p = 4$
(since $S^{-1}E_{a_0} = E_{a_0}\otimes \omega_M^{-1} [-2]$). We
compute each term in the spectral sequence as follows:
\begin{itemize}
\item $p=0$. We have $Ext^4(E_i, E_i\otimes \omega_M^{-1}[-2]) =
  H^2(M, \omega_M^{-1})$ which has dimension $10$ as we have seen
  above and all other groups are $0$. Thus we must have $q=4$ and
  $\dim(\mf{E}_1^{0,4}) = 3 \times 10 = 30$.
\item $p=1$. Then $\dim(Ext^2(E_i,E_j)) = $ is $3$ if $(i,j) = (1,2)$
  or $(2,3)$ and $6$ if $(1,j) = (1,3)$. On the other hand,
  $\dim(Ext^4(E_j,S^{-1}E_i)) = 6$ if $(i,j) = (1,2)$ or $(2,3)$ and
  $3$ if $(i,j) = (1,3)$ and all other groups are $0$. Thus, we must
  have $q = 2+4 = 6$ and each of the three possibilities for $(i,j)$
  contributes a summand of dimension $3 \cdot 6 =6 \cdot 3 = 18$. It
  follows that $\dim(\mf{E}_1^{-1,6}) = 3 \cdot 18 = 54$.
\item $p=2$. We must have $(a_0,a_1,a_2) = (1,2,3)$, so from
  computations in the previous case we must have $q= 2+ 2 + 4 = 8$ and
  $\dim(\mf{E}_1^{-2,8}) = 3 \cdot 3 \cdot 3 = 27$.
\end{itemize}

Since $\mf{E}_1^{0,4}$, $\mf{E}_1^{-1,6}$ and $\mf{E}_1^{-2,8}$ are
the only non-zero terms in the spectral sequence it follows that it
degenerates at $\mf{E}_1$. Consequently, the dimension of
$NHH^t(\me{B}, \me{D}^b(M))$, for $t \geq 0$, are given by the
sequence $0,0,0,0,30,54,27,0,0,\dots$.

From the computations of $HH^{\bullet}(\me{D}^b(M))$ and
$NHH^{\bullet}(\me{B}, \me{D}^b(M))$ above, it follows from
\eqref{eq:seq} that to compute the dimensions of all $HH^t(\me{A})$ it
suffices to compute the rank of the map from $HH^4(\me{B},
\me{D}^b(M))$ to $HH^4(\me{D}^b(M))$. Since $E_1 = O_M$, it follows
from the case $\me{B} = \langle O_X \rangle $, where $X$ is any smooth
projective variety, considered by Kuznetsov \cite[Theorem
8.5]{kuznetsov-hh}, and the functoriality of restriction maps on
Hochschild cohomology for admissible subcategories, that this map is
surjective.

\end{proof}

We note that the dimensions of the Ext groups in our exceptional
collection satisfy the same duality with respect to those for the
exceptional collection $O_{\P^2}(-2), O_{\P^2}(-1),O_{\P^2}$ on $\P^2$
as discussed by Alexeev and Orlov \cite[p.~757]{alexeev-orlov} in the
case of Burniat surfaces.

As mentioned in the introduction, Bloch's conjecture on zero cycles
is still not known for any fake projective plane. Based on standard
motivic conjectures, we make the following
\begin{conj} \label{conj:van}
  Let $X$ be a smooth projective variety over a field $k$ of
  characteristic zero. If $\me{A} \subset \me{D}^b(X)$ is an
  admissible subcategory with $HH_{\bullet}(\me{A}) = 0$ then
  $K_0(\me{A})$ is a torsion group.
\end{conj}

Although the usual motivic conjectures are notoriously intractable, we
hope that the extra structure here might make this more accessible.
If true, together with Theorem \ref{thm:exceptional} it would clearly
imply Bloch's conjecture for the fake projective planes we have
considered.

\begin{rem}
  Because of the existence of $2$-torsion in $P$ for two out of the
  three $2$-adically uniformised fake projective planes, we get more
  exceptional collections in these cases than was conjectured in
  \cite{GKMS}. Moreover, this suggests that for general fake
  projective planes the condition on the existence of cube roots of
  the canonical bundle might be unnecessary.
\end{rem}

\begin{rem}
  Recently Allcock and Kato \cite{allcock-kato} have found a cocompact
  lattice $\Gamma$ in $\mr{PGL}_3(\Q_2)$ containing non-trivial
  torsion such that $\widehat{\Omega}^2/\Gamma$ is still a fake
  projective plane and have suggested that there might be other
  examples as well. Our methods do not immediately apply to their
  example, but the results of \S 2.3 do, so one might expect that a
  more detailed knowledge of the special fibre could lead to the proof
  of the conjecture of Galkin, Katzarkov, Mellit and Shinder in this
  case as well. Of course, for fake projective planes without any
  $p$-adic uniformisation completely new methods would be needed.
\end{rem}
\smallskip

\emph{Acknowledgements.} I thank Gopal Prasad for helpful
correspondence on the classification of fake projective planes and
Patrick Brosnan and Ludmil Katzarkov for their comments on the first
version of this article. I also thank the referee for a very careful
reading of the paper and for his/her helpful comments and corrections.

%\bibliographystyle{siam} 
%\bibliography{../sources}

\def\cprime{$'$}

\end{document}